\documentclass[12pt]{article}
\usepackage{amsmath,amssymb,amsthm}
\numberwithin{equation}{section}
\usepackage{hyperref}
\usepackage{units}
\usepackage{color}
\usepackage[T1]{fontenc}
\usepackage[utf8]{inputenc}
\usepackage{authblk}
\usepackage{bm}
\usepackage{graphicx}

\usepackage{datetime}
\usepackage{color}

\usepackage[toc,page]{appendix}

\newcommand{\Nzero}{{\mathbb N}_{0}}
\newcommand{\Zplus}{{\mathbb Z^{+}}}

\numberwithin{equation}{section}
\newtheorem{thm}{Theorem}[section]

\title{Some remarkable infinite product identities involving Fibonacci and Lucas numbers\thanks{AMS Classification Numbers : 11B39, 11Y60}} 

\author[]{Kunle Adegoke \thanks{kunle.adegoke@yandex.com, adegoke00@gmail.com}}

\affil{Department of Physics and Engineering Physics, \mbox{Obafemi Awolowo University}, Ile-Ife, Nigeria}

\begin{document}

\date{}

\maketitle

\begin{abstract}
\noindent By applying the classic telescoping summation formula and its variants to identities involving inverse hyperbolic tangent functions having inverse powers of the golden ratio as arguments and employing subtle properties of the Fibonacci and Lucas numbers, we derive interesting general infinite product identities involving these numbers.
\end{abstract}

\section{Introduction}
The following {\em striking} (in Frontczak's description~\cite{frontczak}) infinite product identity involving Fibonacci numbers,
\[
\prod_{k = 1}^\infty  {\frac{{F_{2k + 2}  + 1}}{{F_{2k + 2}  - 1}}}  = 3\,,
\]
originally obtained by Melham and Shannon~\cite{melham}, is an $n=1=q$ case of a more general identity (to be derived in this present paper):
\[
\prod_{k = 1}^\infty  {\frac{{F_{(2n - 1)(2k + 2q)}  + F_{2q(2n - 1)} }}{{F_{(2n - 1)(2k + 2q)}  - F_{2q(2n - 1)} }}}  = \prod_{k = 1}^q {\frac{{F_{(2n - 1)(2k - 1)} L_{(2n - 1)2k} }}{{L_{(2n - 1)(2k - 1)} F_{(2n - 1)2k} }}}\,.
\]
The above identity is valid for $q\in\Nzero, n\in\Zplus$, where we use $\Nzero$ to denote the set of natural numbers including zero.

\bigskip

Here the Fibonacci numbers, $F_n$, and Lucas numbers, $L_n$, are defined, for $n\in\Nzero$, as usual, through the recurrence relations \mbox{$F_n=F_{n-1}+F_{n-2}$}, with $F_0=0$, $F_1=1$ and \mbox{$L_n=L_{n-1}+L_{n-2}$}, with $L_0=2$, $L_1=1$.

\bigskip

Our goal in this paper is to derive several general infinite product identities, including the one given above. Our method consists of applying the following general telescoping summation property (equation~(2.1) of \cite{basu}, also~(2.17) of~\cite{gould})
\begin{equation}\label{equ.p2fefzx}
\sum_{k = 1}^N {\left[ {f(k) - f(k + m)} \right]}  = \sum_{k = 1}^m {f(k)}  - \sum_{k = 1}^m {f(k + N)},\quad\mbox{for $N\ge m\ge 1$}\,,
\end{equation}
and its alternating version, obtained by replacing $f(k)$ with $(-1)^{k-1}f(k)$ in~\eqref{equ.p2fefzx}, namely,
\begin{equation}\label{equ.mdgx80r}
\begin{split}
&\sum_{k = 1}^N {( - 1)^{k - 1} \left[ {f(k) + ( - 1)^{m - 1} f(k + m)} \right]}\\ 
&\quad= \sum_{k = 1}^m {( - 1)^{k - 1} f(k)}  + ( - 1)^{N - 1} \sum_{k = 1}^m {( - 1)^{k - 1} f(k + N)} 
\end{split}
\end{equation}
to some inverse hyperbolic tangent identities with the inverse powers of the golden mean as arguments of the inverse hyperbolic tangent functions.

\bigskip

If $f(N)\to 0$ as $N\to\infty$, then we have, from~\eqref{equ.p2fefzx} and~\eqref{equ.mdgx80r}, the useful identities
\begin{equation}\label{equ.gy1asjs}
\sum_{k = 1}^\infty  {\left[ {f(k) - f(k + m)} \right]}  = \sum_{k = 1}^m {f(k)} 
\end{equation}
and
\begin{equation}\label{equ.vdz8jvp}
\sum_{k = 1}^\infty  {( - 1)^{k - 1} \left[ {f(k) + ( - 1)^{m - 1} f(k + m)} \right]}  = \sum_{k = 1}^m {( - 1)^{k - 1} f(k)}\,.
\end{equation}
In handling infinite products, the following identity will often come in handy 
\begin{equation}\label{equ.qvxo2rv}
\prod_{k = 1}^m {f(k)}  = \prod_{k = 1}^{\left\lceil {m/2} \right\rceil } {f(2k - 1)} \prod_{k = 1}^{\left\lfloor {m/2} \right\rfloor } {f(2k)}\,,
\end{equation}
where $\lfloor u\rfloor$ is the greatest integer less than or equal to $u$ and $\lceil u\rceil$ is the smallest integer greater than or equal to $u$.

\bigskip

Specifically, we have
\[
\prod_{k=1}^{2q}f(k)=\prod_{k=1}^q{f(2k-1)f(2k)}
\]
and
\[
\prod_{k=1}^{2q-1}f(k)=\prod_{k=1}^q{f(2k-1)}\prod_{k=1}^{q-1}{f(2k)}\,.
\]
We shall adopt the following empty product convention:
\[\prod_{k=1}^0 f(k)=1\,.\]
The golden ratio, having the numerical value of $(\sqrt 5+1)/2$, denoted throughout this paper by $\phi$, will appear frequently. We shall often make use of the following algebraic properties of $\phi$

\begin{subequations}
\begin{eqnarray}\label{equ.xj19lbd}
\phi^2 &=& 1+\phi\,,\label{equ.zuch1o9}\\
\sqrt 5 &=& 2\phi-1\,,\label{equ.ln4di53}\\
\phi-1 &=& 1/\phi\,,\label{equ.epgagl8}\\
\phi^n &=& \phi F_n+F_{n-1}\,,\label{equ.n1xmwpl}\\
\phi^{-n} &=& (-1)^n(-\phi F_n+F_{n+1})\,,\label{equ.wr334af}\\
\mbox{and}\nonumber\\
\phi^n &=& \phi^{n-1}+\phi^{n-2}\label{equ.gxbvppb}\,.
\end{eqnarray}
\end{subequations}

\bigskip

Using Binet's formula, $F_n\sqrt 5=\phi^n-(-\phi)^{-n}$,~$n\in\mathbb{Z}$ and identities~\eqref{equ.n1xmwpl} and \eqref{equ.wr334af}, it is straightforward to establish the following useful identities
\begin{subequations}\label{equ.z0ctl7a}
\begin{eqnarray}
\phi^{n}-\phi^{-n} &=& F_n\sqrt 5,\quad\mbox{$n$ even}\label{equ.myu8b1u}\\
\phi^{n}+\phi^{-n} &=& L_n,\quad\mbox{$n$ even}\\
\phi^{n}+\phi^{-n} &=& F_n\sqrt 5,\quad\mbox{$n$ odd}\\
\phi^{n}-\phi^{-n} &=& L_n,\quad\mbox{$n$ odd}
\end{eqnarray}
\end{subequations}
From~\eqref{equ.z0ctl7a} we also have the following useful results
\begin{subequations}\label{equ.zdqegff}
\begin{eqnarray}
\frac{\phi^n+1}{\phi^n-1}&=&\frac{F_n\sqrt 5+2}{L_n},\quad \mbox{$n$ odd,}\\
\frac{\phi^n+1}{\phi^n-1}&=&\frac{F_n\sqrt 5}{L_n-2},\quad \mbox{$n$ even}\,.
\end{eqnarray}
\end{subequations}

\bigskip

We shall make repeated use of the following identities connecting Fibonacci and Lucas numbers:
\begin{subequations}\label{equ.u4fy3q3}
\begin{eqnarray}\label{equ.wsr7wzv}
F_{2n}=F_nL_n\,,\\
L_{2n}-2(-1)^n=5F_n^2\label{equ.s6sil39}\,,\\
5F_n^2-L_n^2=4(-1)^{(n+1)}\,,\\
L_{2n}+2(-1)^n=L_n^2\label{equ.wle2jmq}\,.
\end{eqnarray}
\end{subequations}
Identities~\eqref{equ.u4fy3q3} or their variations can be found in~\cite{basin2, dunlap, howard}. 

\bigskip

Finally, the relevant inverse hyperbolic function identities that we shall require are the following 
\begin{subequations}
\begin{eqnarray}
\tanh^{-1} x+\tanh^{-1} y &=& \tanh^{-1}\left (\frac{x+y}{1+xy}\right )\,,\quad xy<1\label{equ.ohqvcqn}\\
\tanh^{-1} x-\tanh^{-1} y &=& \tanh^{-1}\left (\frac{x-y}{1-xy}\right )\,,\quad xy>-1\,\label{equ.gdoxus5},\\
\tanh^{-1}\left(\frac{x}{y}\right)&=&\frac{1}{2}\log\left(\frac{y+x}{y-x}\right),\quad |x|<|y|,\label{equ.clyqxe1}\\
\tanh^{-1}\left(\frac{y}{x}\right)&=&\tanh^{-1}\left(\frac{x}{y}\right)-i\frac{\pi}{2},\quad |x|<|y|\label{equ.k7jhoz9}\,.
\end{eqnarray} 
\end{subequations}
Infinite product identities involving Fibonacci numbers and Lucas numbers were also derived in~\cite{frontczak,melham,melham2} and references therein.
\section{Infinite product identities}
\subsection{Identities resulting from taking $f(k)=\tanh^{-1}(\phi^{-2pk})$~in~\eqref{equ.gy1asjs} }
\begin{thm}
For $q,n\in\Zplus$ the following infinite product identities hold:
\begin{equation}\label{equ.dh73nut}
\prod_{k = 1}^\infty  {\frac{{L_{(2n - 1)(2k + 2q - 1)}  + L_{(2n - 1)(2q - 1)} }}{{L_{(2n - 1)(2k + 2q - 1)}  - L_{(2n - 1)(2q - 1)} }}}  = \sqrt 5 \prod_{k = 1}^q {\frac{{F_{(2k - 1)(2n - 1)} }}{{L_{(2k - 1)(2n - 1)} }}} \prod_{k = 1}^{q - 1} {\frac{{L_{2k(2n - 1)} }}{{F_{2k(2n - 1)} }}}\,,
\end{equation}
\begin{equation}\label{equ.u36rrvi}
\prod_{k = 1}^\infty  {\frac{{F_{2n(2k + 2q - 1)}  + F_{2n(2q - 1)} }}{{F_{2n(2k + 2q - 1)}  - F_{2n(2q - 1)} }}}  = \frac{1}{{(\sqrt 5 )^{2q - 1} }}\prod_{k = 1}^{2q - 1} {\frac{{L_{2nk} }}{{F_{2nk} }}}\,,
\end{equation}
\begin{equation}\label{equ.if6cwjv}
\prod_{k = 1}^\infty  {\frac{{F_{4n(k + q)}  + F_{4nq} }}{{F_{4n(k + q)}  - F_{4nq} }}}  = \frac{1}{{5^q }}\prod_{k = 1}^{2q} {\frac{{L_{2nk} }}{{F_{2nk} }}}
\end{equation}
and
\begin{equation}\label{equ.unbkp7a}
\prod_{k = 1}^\infty  {\frac{{F_{(2n - 1)(2k + 2q)}  + F_{2q(2n - 1)} }}{{F_{(2n - 1)(2k + 2q)}  - F_{2q(2n - 1)} }}}  = \prod_{k = 1}^q {\frac{{F_{(2n - 1)(2k - 1)} L_{(2n - 1)2k} }}{{L_{(2n - 1)(2k - 1)} F_{(2n - 1)2k} }}}\,.
\end{equation}

\end{thm}

\begin{proof}
Choosing  $f(k) = \tanh ^{ - 1} (\phi ^{-2pk}) $ in~\eqref{equ.gy1asjs} and using~\eqref{equ.gdoxus5} and~\eqref{equ.z0ctl7a}, we obtain
\begin{equation}\label{equ.aidawqg}
\sum_{k = 1}^\infty  {\tanh ^{ - 1} \left[ {\frac{{L_{pm} }}{{L_{p(2k + m)} }}} \right]}  = \sum_{k = 1}^m {\tanh ^{ - 1} \left[ {\frac{1}{{\phi ^{2pk} }}} \right]},\quad\mbox{$p\,m$ odd}
\end{equation}
and
\begin{equation}\label{equ.i2hcpbw}
\sum_{k = 1}^\infty  {\tanh ^{ - 1} \left[ {\frac{{F_{pm} }}{{F_{p(2k + m)} }}} \right]}  = \sum_{k = 1}^m {\tanh ^{ - 1} \left[ {\frac{1}{{\phi ^{2pk} }}} \right]},\quad\mbox{$p\,m$ even}\,.
\end{equation}
Setting $p=2n-1$ and $m=2q-1$ in~\eqref{equ.aidawqg} and converting the infinite sum identity to an infinite product identity, employing also the identities~\eqref{equ.zdqegff}, \eqref{equ.s6sil39}, \eqref{equ.wle2jmq} and~\eqref{equ.qvxo2rv}, we have
\begin{equation}
\begin{split}
&\prod_{k = 1}^\infty  {\frac{{L_{(2n - 1)(2k + 2q - 1)}  + L_{(2n - 1)(2q - 1)} }}{{L_{(2n - 1)(2k + 2q - 1)}  - L_{(2n - 1)(2q - 1)} }}}\\ 
&\qquad= \prod_{k = 1}^{2q - 1} {\frac{{F_{2k(2n - 1)} \sqrt 5 }}{{L_{2k(2n - 1)}  - 2}}}\\ 
&\qquad= (\sqrt 5 )^{2q - 1} \prod_{k = 1}^{2q - 1} {\frac{{F_{2k(2n - 1)} }}{{L_{2k(2n - 1)}  - 2}}}\\
 &\qquad= (\sqrt 5 )^{2q - 1} \prod_{k = 1}^q {\frac{{F_{2(2k - 1)(2n - 1)} }}{{L_{2(2k - 1)(2n - 1)}  - 2}}} \prod_{k = 1}^{q - 1} {\frac{{F_{4k(2n - 1)} }}{{L_{4k(2n - 1)}  - 2}}}\\ 
 &\qquad= (\sqrt 5 )^{2q - 1} \prod_{k = 1}^q {\frac{{F_{(2k - 1)(2n - 1)} L_{(2k - 1)(2n - 1)} }}{{L_{(2k - 1)(2n - 1)}^2 }}} \prod_{k = 1}^{q - 1} {\frac{{F_{2k(2n - 1)} L_{2k(2n - 1)} }}{{5F_{2k(2n - 1)}^2 }}}\\ 
 &\qquad= \sqrt 5 \prod_{k = 1}^q {\frac{{F_{(2k - 1)(2n - 1)} }}{{L_{(2k - 1)(2n - 1)} }}} \prod_{k = 1}^{q - 1} {\frac{{L_{2k(2n - 1)} }}{{F_{2k(2n - 1)} }}}\,, 
\end{split}
\end{equation}

which proves~\eqref{equ.dh73nut}.

\bigskip

The possibilities for the choice of $m$ and $p$ in~\eqref{equ.i2hcpbw} are \mbox{(i) $p$ even, $m$ odd}, \mbox{(ii) $p$ even, $m$ even} and \mbox{(iii) $p$ odd, $m$ even}. Converting the infinite sums in~\eqref{equ.i2hcpbw} to infinite products and setting $p=2n$ and $m=2q-1$ proves identity~\eqref{equ.u36rrvi}; while the identity~\eqref{equ.if6cwjv} is proved by setting $p=2n$ and $m=2q$. Finally, identity~\eqref{equ.unbkp7a} is proved by setting $p=2n-1$ and $m=2q$ in the infinite product identity resulting from~~\eqref{equ.i2hcpbw}.

\end{proof}

\subsection{Identities resulting from taking $f(k)=\tanh^{-1}(\phi^{-p(2k-1)})$~in~\eqref{equ.gy1asjs} }
\begin{thm}
For $q,n\in\Zplus$ the following infinite product identities hold:
\begin{equation}\label{equ.ujzmeac}
\prod_{k = 1}^\infty  {\frac{{F_{4n(2k + q - 1)}  + F_{4nq} }}{{F_{4n(2k + q - 1)}  - F_{4nq} }}}  = \frac{1}{{\sqrt 5 ^q }}\prod_{k = 1}^q {\frac{{L_{2n(2k - 1)} }}{{F_{2n(2k - 1)} }}}\,,
\end{equation}
\begin{equation}
\prod_{k = 1}^\infty  {\frac{{F_{(4n - 2)(2k + q - 1)}  + F_{2q(2n - 1)} }}{{F_{(4n - 2)(2k + q - 1)}  - F_{2q(2n - 1)} }}}  = \sqrt 5 ^q \prod_{k = 1}^q {\frac{{F_{(2n - 1)(2k - 1)} }}{{L_{(2n - 1)(2k - 1)} }}}\,,
\end{equation}
\begin{equation}\label{equ.ufuabxd}
\prod_{k = 1}^\infty  {\frac{{L_{(2k - 1 + 2q)(2n - 1)}  + \sqrt 5 F_{2q(2n - 1)} }}{{L_{(2k - 1 + 2q)(2n - 1)}  - \sqrt 5 F_{2q(2n - 1)} }}}  = \prod_{k = 1}^{2q} {\frac{{F_{(2k - 1)(2n - 1)} \sqrt 5  + 2}}{{L_{(2k - 1)(2n - 1)} }}}
\end{equation}
and
\begin{equation}\label{equ.qg32f22}
\prod_{k = 1}^\infty  {\frac{{\sqrt 5 F_{2(2n-1)(k + q - 1)}  + L_{(2n - 1)(2q - 1)} }}{{\sqrt 5 F_{2(2n-1)(k + q - 1)}  - L_{(2n - 1)(2q - 1)} }}}  = \prod_{k = 1}^{2q - 1} {\frac{{F_{(2k - 1)(2n - 1)} \sqrt 5  + 2}}{{L_{(2k - 1)(2n - 1)} }}}\,.
\end{equation}

\end{thm}

\begin{proof}
Choosing $f(k)=\tanh^{-1}(\phi^{-p(2k-1)})$~in~\eqref{equ.gy1asjs} and proceeding as in the previous section we obtain the following infinite sum identities:
\[
\sum_{k = 1}^\infty  {\tanh ^{ - 1} \left[ {\frac{{F_{mp} }}{{F_{2kp + mp - p} }}} \right]}  = \sum_{k = 1}^m {\tanh ^{ - 1} \frac{1}{{\phi ^{p(2k - 1)} }}}\,,\quad\mbox{$p$ even}\,,
\]
\[
\sum_{k = 1}^\infty  {\tanh ^{ - 1} \left[ {\frac{{F_{mp} \sqrt 5 }}{{L_{2kp + mp - p} }}} \right]}  = \sum_{k = 1}^m {\tanh ^{ - 1} \frac{1}{{\phi ^{p(2k - 1)} }}}\,,\quad\mbox{$p$ odd, $m$ even}\,, 
\]
and
\[
\sum_{k = 1}^\infty  {\tanh ^{ - 1} \left[ {\frac{{L_{mp} }}{{\sqrt 5 F_{2kp + mp - p} }}} \right]}  = \sum_{k = 1}^m {\tanh ^{ - 1} \frac{1}{{\phi ^{p(2k - 1)} }}}\,,\quad\mbox{$p\,m$ odd}\,. 
\]
Identities~\eqref{equ.ujzmeac} --- \eqref{equ.qg32f22} are established by converting each infinite sum identity to an infinite product identity with the appropriate choice of $m$ and $p$ in each case.

\end{proof}

In closing this section we observe that $q=1=n$ in~\eqref{equ.ufuabxd} and \eqref{equ.qg32f22} yield the following special evaluations:
 \[
\prod\limits_{k = 1}^\infty  {\frac{{L_{2k + 1}  + \sqrt 5 }}{{L_{2k + 1}  - \sqrt 5 }}}  = \phi ^4\, ,\quad\prod\limits_{k = 1}^\infty  {\frac{{\sqrt 5 F_{2k}  + 1}}{{\sqrt 5 F_{2k}  - 1}}}  = \phi ^3\,. 
\]

\subsection{Identities resulting from taking $f(k)=\tanh^{-1}(\phi^{-2pk})$~in~\eqref{equ.vdz8jvp} }
\begin{thm}
The following infinite product identities hold for $q,n\in\Zplus$:
\begin{equation}\label{equ.nm7xvk2}
\prod_{k = 1}^\infty  {\frac{{F_{4nq + 4nk}  + ( - 1)^{k - 1} F_{4nq} }}{{F_{4nq + 4nk}  + ( - 1)^k F_{4nq} }}}  = \prod_{k = 1}^q {\frac{{L_{2n(2k - 1)} F_{4nk} }}{{F_{2n(2k - 1)} L_{4nk} }}}\,,
\end{equation}
\begin{equation}\label{equ.p5moynn}
\prod_{k = 1}^\infty  {\frac{{F_{(2n - 1)(2q + 2k)}  + ( - 1)^{k - 1} F_{(2n - 1)2q} }}{{F_{(2n - 1)(2q + 2k)}  + ( - 1)^k F_{(2n - 1)2q} }}}  = 5^q \prod_{k = 1}^q {\frac{{F_{(2n - 1)(2k - 1)} F_{(2n - 1)2k} }}{{L_{(2n - 1)(2k - 1)} L_{(2n - 1)2k} }}}\,,
\end{equation}
\begin{equation}\label{equ.nvho6p6}
\prod_{k = 1}^\infty  {\frac{{L_{4nk + 4nq - 2n}  + ( - 1)^{k - 1} L_{4nq - 2n} }}{{L_{4nk + 4nq - 2n}  + ( - 1)^k L_{4nq - 2n} }}}  = \frac{1}{{\sqrt 5 }}\prod_{k = 1}^q {\frac{{L_{2n(2k - 1)} }}{{F_{2n(2k - 1)} }}} \prod_{k = 1}^{q - 1} {\frac{{F_{4nk} }}{{L_{4nk} }}}
\end{equation}
and
\begin{equation}\label{equ.tu8kmxl}
\begin{split}
&\prod_{k = 1}^\infty  {\frac{{F_{(2n - 1)(2q + 2k - 1)}  + ( - 1)^{k - 1} F_{(2n - 1)(2q - 1)} }}{{F_{(2n - 1)(2q + 2k - 1)}  + ( - 1)^k F_{(2n - 1)(2q - 1)} }}}\\ 
&\qquad\qquad= (\sqrt 5 )^{2q - 1} \prod_{k = 1}^q {\frac{{F_{(2n - 1)(2k - 1)} }}{{L_{(2n - 1)(2k - 1)} }}} \prod_{k = 1}^{q - 1} {\frac{{F_{(2n - 1)2k} }}{{L_{(2n - 1)2k} }}}\,.
\end{split}
\end{equation}

\end{thm}

\begin{proof}
Taking $f(k)=\tanh^{-1}(\phi^{-2pk})$~in~\eqref{equ.vdz8jvp}, setting $m=2q$ and making use of the identities~\eqref{equ.gdoxus5} and~\eqref{equ.myu8b1u} we have
\[
\sum_{k = 1}^\infty  {( - 1)^{k - 1} \tanh ^{ - 1} \left[ {\frac{{F_{2qp} }}{{F_{2kp + 2qp} }}} \right]}  = \sum_{k = 1}^{2q} {( - 1)^{k - 1} \tanh ^{ - 1} \frac{1}{{\phi ^{2pk} }}}\,,
\]
from which follows
\begin{equation}\label{qu.ub8uzh3}
\begin{split}
\prod_{k = 1}^\infty  {\frac{{F_{2pq + 2pk}  + ( - 1)^{k - 1} F_{2pq} }}{{F_{2pq + 2pk}  + ( - 1)^k F_{2pq} }}}  &= \prod_{k = 1}^{2q} {\frac{{\phi ^{2pk}  + ( - 1)^{k - 1} }}{{\phi ^{2pk}  + ( - 1)^k }}}\\ 
&= \prod_{k = 1}^q {\frac{{\phi ^{2p(2k - 1)}  + 1}}{{\phi ^{2p(2k - 1)}  - 1}}} \prod_{k = 1}^q {\frac{{\phi ^{4pk}  - 1}}{{\phi ^{4pk}  + 1}}}\\ 
& = \prod_{k = 1}^q {\frac{{F_{2p(2k - 1)} \sqrt 5 }}{{L_{2p(2k - 1)}  - 2}}\frac{{L_{4pk}  - 2}}{{F_{4pk} \sqrt 5 }}}\\ 
& = 5^q \prod_{k = 1}^q {\frac{{F_{2p(2k - 1)} }}{{L_{2p(2k - 1)}  - 2}}\frac{{F_{2pk} }}{{L_{2pk} }}}\,. 
\end{split}
\end{equation}
The final evaluation of the denominator of the product term in the last line of identity~\eqref{qu.ub8uzh3} depends on the parity of $p$. Setting $p=2n$ in~\eqref{qu.ub8uzh3}, making use of~\eqref{equ.s6sil39} proves~\eqref{equ.nm7xvk2} while setting $p=2n-1$ in~\eqref{qu.ub8uzh3}, utilizing~\eqref{equ.wle2jmq} proves~\eqref{equ.p5moynn}.

Taking $f(k)=\tanh^{-1}(\phi^{-2pk})$~in~\eqref{equ.vdz8jvp}, setting $m=2q-1$ and making use of the identities~\eqref{equ.ohqvcqn} and~\eqref{equ.z0ctl7a} we have
\[
\sum_{k = 1}^\infty  {( - 1)^{k - 1} \tanh ^{ - 1} \left[ {\frac{{L_{2qp - p} }}{{L_{2kp + 2qp - p} }}} \right]}  = \sum_{k = 1}^{2q - 1} {( - 1)^{k - 1} \tanh ^{ - 1} \frac{1}{{\phi ^{2pk} }}}\,,\quad\mbox{$p$ even}\,, 
\]
and
\[
\sum_{k = 1}^\infty  {( - 1)^{k - 1} \tanh ^{ - 1} \left[ {\frac{{F_{2qp - p} }}{{F_{2kp + 2qp - p} }}} \right]}  = \sum_{k = 1}^{2q - 1} {( - 1)^{k - 1} \tanh ^{ - 1} \frac{1}{{\phi ^{2pk} }}}\,,\quad\mbox{$p$ odd}\,, 
\]
from which identities~\eqref{equ.nvho6p6} and~\eqref{equ.tu8kmxl} follow immediately.

\end{proof}

\subsection{Identities resulting from taking $f(k)=\tanh^{-1}(\phi^{-p(2k-1)})$~in~\eqref{equ.vdz8jvp} }
\begin{thm}
The following infinite product identities hold for $q,n\in\Zplus$:
\begin{equation}\label{equ.eb4ic5q}
\prod_{k = 1}^\infty  {\frac{{F_{8nk + 8nq - 4n}  + ( - 1)^{k - 1} F_{8nq} }}{{F_{8nk + 8nq - 4n}  + ( - 1)^k F_{8nq} }}}  = \prod_{k = 1}^q {\frac{{L_{2n(4k - 3)} F_{2n(4k - 1)} }}{{F_{2n(4k - 3)} L_{2n(4k - 1)} }}}\,,
\end{equation}
\begin{equation}\label{equ.sy0qp2p}
\prod_{k = 1}^\infty  {\frac{{F_{(4n - 2)(2q + 2k - 1)}  + ( - 1)^{k - 1} F_{(2n - 1)4q} }}{{F_{(4n - 2)(2q + 2k - 1)}  + ( - 1)^k F_{(2n - 1)4q} }}}  = \prod_{k = 1}^q {\frac{{F_{(2n - 1)(4k - 3)} L_{(2n - 1)(4k - 1)} }}{{L_{(2n - 1)(4k - 3)} F_{(2n - 1)(4k - 1)} }}}\,,
\end{equation}
\begin{equation}\label{equ.nj41t02}
\begin{split}
&\prod_{k = 1}^\infty  {\frac{{L_{(2n - 1)(2q + 2k - 1)}  + ( - 1)^{k - 1} \sqrt 5 F_{(2n - 1)2q} }}{{L_{(2n - 1)(2q + 2k - 1)}  + ( - 1)^k \sqrt 5 F_{(2n - 1)2q} }}}\\  
&\qquad\qquad= \prod_{k = 1}^q {\frac{{(F_{(2n - 1)(4k - 3)} \sqrt 5  + 2)L_{(2n - 1)(4k - 1)} }}{{(F_{(2n - 1)(4k - 1)} \sqrt 5  + 2)L_{(2n - 1)(4k - 3)} }}}\,,
\end{split}
\end{equation}
\begin{equation}\label{equ.s62fvgq}
\prod_{k = 1}^\infty  {\frac{{L_{8n(k + q - 1)}  + ( - 1)^{k - 1} L_{4n(2q - 1)} }}{{L_{8n(k + q - 1)}  + ( - 1)^k L_{4n(2q - 1)} }}}  = \frac{1}{{\sqrt 5 }}\prod_{k = 1}^q {\frac{{L_{2n(4k - 3)} }}{{F_{2n(4k - 3)} }}} \prod_{k = 1}^{q - 1} {\frac{{F_{2n(4k - 1)} }}{{L_{2n(4k - 1)} }}}\,,
\end{equation}
\begin{equation}\label{equ.kmp83dm}
\begin{split}
&\prod_{k = 1}^\infty  {\frac{{L_{(8n - 4)(k + q - 1)}  + ( - 1)^{k - 1} L_{(4n - 2)(2q - 1)} }}{{L_{(8n - 4)(k + q - 1)}  + ( - 1)^k L_{(4n - 2)(2q - 1)} }}}\\ 
&\qquad\qquad= \sqrt 5 \prod_{k = 1}^q {\frac{{F_{(2n - 1)(4k - 3)} }}{{L_{(2n - 1)(4k - 3)} }}} \prod_{k = 1}^{q - 1} {\frac{{L_{(2n - 1)(4k - 1)} }}{{F_{(2n - 1)(4k - 1)} }}} 
\end{split}
\end{equation}
and
\begin{equation}\label{equ.f7wwhwq}
\begin{split}
&\prod_{k = 1}^\infty  {\frac{{L_{(4n - 2)(q + k - 1)}  + ( - 1)^{k - 1} \sqrt 5 F_{(2n - 1)(2q - 1)} }}{{L_{(4n - 2)(q + k - 1)}  + ( - 1)^k \sqrt 5 F_{(2n - 1)(2q - 1)} }}}\\
&\qquad\qquad=\prod_{k = 1}^q {\frac{{(F_{(2n - 1)(4k - 3)} \sqrt 5  + 2)}}{{L_{(2n - 1)(4k - 3)} }}} \prod_{k = 1}^{q - 1} {\frac{{L_{(2n - 1)(4k - 1)} }}{{(F_{(2n - 1)(4k - 1)} \sqrt 5  + 2)}}}\,.
\end{split}
\end{equation}

\end{thm}

\begin{proof}
Taking $f(k)=\tanh^{-1}(\phi^{-p(2k-1)})$~in~\eqref{equ.vdz8jvp} and setting $m=2q$ we obtain the following summation identities 
\begin{equation}\label{equ.y84rd8p}
\sum_{k = 1}^\infty  {( - 1)^{k - 1} \tanh ^{ - 1} \left[ {\frac{{F_{2qp} }}{{F_{2kp + 2qp - p} }}} \right]}  = \sum_{k = 1}^{2q} {( - 1)^{k - 1} \tanh ^{ - 1} \frac{1}{{\phi ^{2pk - p} }}}\,,\quad\mbox{$p$ even}\,, 
\end{equation}
and
\begin{equation}\label{equ.zraml30}
\sum_{k = 1}^\infty  {( - 1)^{k - 1} \tanh ^{ - 1} \left[ {\frac{{F_{2qp} \sqrt 5 }}{{L_{2kp + 2qp - p} }}} \right]}  = \sum_{k = 1}^{2q} {( - 1)^{k - 1} \tanh ^{ - 1} \frac{1}{{\phi ^{2pk - p} }}}\,,\quad\mbox{$p$ odd} \,.
\end{equation}
Identities~\eqref{equ.eb4ic5q} and~\eqref{equ.sy0qp2p} follow from~\eqref{equ.y84rd8p} while the identity~\eqref{equ.nj41t02} follows directly from~\eqref{equ.zraml30}.

\bigskip

Taking $f(k)=\tanh^{-1}(\phi^{-p(2k-1)})$~in~\eqref{equ.vdz8jvp} and setting $m=2q-1$ we obtain the following summation identities
\begin{equation}\label{equ.tus76qt}
\sum_{k = 1}^\infty  {( - 1)^{k - 1} \tanh ^{ - 1} \left[ {\frac{{L_{p(2q - 1)} }}{{L_{2p(k + q - 1)} }}} \right]}  = \sum_{k = 1}^{2q - 1} {( - 1)^{k - 1} \tanh ^{ - 1} \frac{1}{{\phi ^{2pk - p} }}}\,,\quad\mbox{$p$ even} 
\end{equation}
and
\begin{equation}\label{equ.gjgj80e}
\sum_{k = 1}^\infty  {( - 1)^{k - 1} \tanh ^{ - 1} \left[ {\frac{{F_{p(2q - 1)} \sqrt 5 }}{{L_{2p(k + q - 1)} }}} \right]}  = \sum_{k = 1}^{2q - 1} {( - 1)^{k - 1} \tanh ^{ - 1} \frac{1}{{\phi ^{2pk - p} }}}\,,\quad\mbox{$p$ odd}\,. 
\end{equation}
Identities~\eqref{equ.s62fvgq} and~\eqref{equ.kmp83dm} follow from evaluating the summation identity~\eqref{equ.tus76qt} at $p=4n$ and $p=4n-2n$, respectively. Identity~\eqref{equ.f7wwhwq} is obtained by converting the summation identity~\eqref{equ.gjgj80e} to an infinite product identity at $p=2n-1$.

\end{proof}

From~\eqref{equ.nj41t02} and \eqref{equ.f7wwhwq} at $q=1=n$ come the special evaluations 
\[
\prod\limits_{k = 1}^\infty  {\frac{{L_{2k + 1}  + ( - 1)^{k - 1} \sqrt 5 }}{{L_{2k + 1}  + ( - 1)^k \sqrt 5 }}}  = \phi ^2\,,\quad\prod\limits_{k = 1}^\infty  {\frac{{L_{2k}  + ( - 1)^{k - 1} \sqrt 5 }}{{L_{2k}  + ( - 1)^k \sqrt 5 }}}  = \phi ^3\,. 
\]


\end{document}